\documentclass[11pt]{amsart}
\usepackage[margin=1in]{geometry}
\usepackage[T1]{fontenc}
\usepackage[utf8]{inputenc}
\usepackage{microtype}
\usepackage{mathtools,amssymb,amsthm}
\usepackage{tikz}
\usetikzlibrary{arrows.meta}
\usepackage[hidelinks]{hyperref}
\theoremstyle{plain}
\newtheorem{theorem}{Theorem}[section]
\newtheorem{proposition}[theorem]{Proposition}
\newtheorem{lemma}[theorem]{Lemma}
\newtheorem{corollary}[theorem]{Corollary}
\theoremstyle{definition}
\newtheorem{definition}[theorem]{Definition}
\theoremstyle{remark}
\newtheorem{remark}[theorem]{Remark}

\newcommand{\R}{\mathbb R}

\newcommand{\cM}{\mathcal M}
\newcommand{\cN}{\mathcal N}

\newcommand{\inner}[2]{\langle #1,#2\rangle}

\title[On the trapping of ray families by mirrors]{On the trapping of ray families by mirrors}
\author{Casey O'Malley}
\address{University of California, Riverside, CA 92521, USA}
\email{comal002@ucr.edu}
\date{}


\begin{document}

\begin{abstract}
We settle two questions of Serge Tabachnikov regarding the trapping of light rays by mirrors. We show (i) that the largest dimension of a trapped family of rays in $\mathbb{R}^n$ is $2n-3$, and (ii) that in $\mathbb{R}^3$ one can trap a 2-parameter family of rays that is not locally normal to any smooth surface. We first establish that no $2n-2$-parameter family can be trapped by showing that no open subset of $L_n$ can be trapped due to Poincaré's recurrence theorem, yielding an upper bound of $2n-3$. We then show that a system of mirrors traps a $2n-3$ parameter family of rays by noting that a compact subset of the set of states of rays in the central "waist" of the trap is a normally hyperbolic invariant manifold, making it amenable to a stable manifold theorem for NHIMs, which implies that the set of rays asymptotic to the waist of the trap has dimension $2n-3$. We then truncate our mirror system, which permits the entrance of rays from arbitrarily far outside of the trap, attaining the upper bound $2n-3$ and resolving the first question. We then show that in $\mathbb{R}^3$, there exists within the trapped family of rays a 2-parameter subfamily on which the canonical symplectic form on the space of oriented lines is nowhere vanishing, implying that the subfamily is not locally normal to any smooth surface, which settles the second question.
\end{abstract}
\maketitle
\section{Introduction}
The trapping of light rays by systems of mirrors is a subject of dynamical systems that originates with John Connett's 1992 question of whether light rays could be trapped within a system of mirrors, which was answered affirmatively with a construction of Roberto Peirone in 1994. In a 2022 survey of open problems in billiards and geometric optics~\cite{BFGLPT}, Serge Tabachnikov asks the following questions:

\medskip\noindent\textbf{Question 1.} \emph{What is the largest dimension of a family of rays in $\R^n$ that can be trapped?}

\medskip\noindent\textbf{Question 2.} \emph{In $\R^3$, can one trap a non-normal 2-parameter family of rays?}

\medskip

We naturally say a ray is trapped if it enters from arbitrarily far (infinitely far, if you wish) outside of some bounded region, i.e. the trap, into the trap, with all of its forward reflections inside of the trap, and we assume all reflections to be nonsingular. Note that $L_n$, the space of oriented lines in $\R^n$, has dimension $2n-2$. We introduce two theorems:

\begin{theorem}\label{thm:main}
For every $n\ge2$:
\begin{enumerate}
    \item No nonempty open subset of $L_n$ is trapped.
    \item There exists a trap that traps a family of rays of dimension $2n-3$ in $L_n$.
\end{enumerate}
Thus the largest possible dimension of a trapped family of rays is $2n-3$.
\end{theorem}

\begin{theorem}\label{thm:nonnormal}
In $\R^3$, there exists a trapped non-normal 2-parameter family of rays.
\end{theorem}

We begin by establishing an upper bound of $2n-3$ by showing that a trapped open subset of $L_n$ is impossible by Poincaré's recurrence theorem, proving (1) of Theorem~\ref{thm:main}. To attain the bound, we then show that an inward-facing mirror formed by a tube of revolution
\[
\rho\in C^\infty(\R,(0,\infty)),
\quad
\rho(0)>0,
\quad
\rho'(0)=0,
\quad
\rho''(0)>0,
\]
truncated to $|y|=\rho(x)$ (with $(x,y)\in\R\times\R^{n-1}$) at $|x|=a$ traps a $(2n-3)$-parameter family of rays. The simplest example of our trap resembles an open-ended hourglass (Figure~\ref{fig:hourglass}), an open tube with a thin central waist toward which some rays converge and become trapped (i.e. they reflect off of some point inside of the trap, and subsequent reflections get closer and closer to the waist). Other examples include catenoids. After showing that the linearization of the billiard map in the directions transverse to the waist is hyperbolic, and that a compact subset of the set of states of rays in the waist is a normally hyperbolic invariant manifold, we apply a stable manifold theorem for normally hyperbolic invariant manifolds to show that a family is trapped.

We then apply our result to show that a non-normal 2-parameter family can be trapped in $\R^3$. We note that a differential form on a 2-parameter family is zero if the family is normal to some surface, and then we show that the form is nonzero for the set of rays within the waist of our trap.

\section{Ray space and trapping}\label{sec:setup}

We recall that the space of oriented lines $L_n=\{(u,p):u\in S^{n-1},\ p\in u^\perp\}$ has $\dim L_n=2n-2,$ and give two definitions.

\begin{definition}
A \emph{trap} is a finite collection of compact embedded $C^\infty$ hypersurfaces (the mirrors) in $\R^n$. Rays are reflected specularly at interior points of the mirrors. We call a trajectory \emph{nonsingular} if it has no tangential collisions with the mirrors, collisions with the boundaries of the mirrors (each mirror may have a boundary), multiple collisions, or accumulation of collisions in finite time.
\end{definition}

\begin{definition}\label{def:trapped}
Let $F\subset L_n$ be a smooth embedded submanifold. We call $F$ a \emph{trapped family of rays} if there is some ball $B_R$ containing a trap in its interior such that the billiard trajectories entering $B_R$ along any $\ell\in F$ from infinity remain within it for all subsequent time and are nonsingular.
\end{definition}

\section{upper bound of the dimension of a trapped family}\label{sec:upper}

\begin{proposition}\label{prop:no-open}
No nonempty open subset of $L_n$ is trapped.
\end{proposition}

\begin{proof}
Suppose that some nonempty open set of rays $U\subset L_n$ is trapped in $B_R$, and let $E$ be the set of entry states of rays in $U$ into $B_R$ along $\partial B_R$. Then make $\partial B_R$ an inward-facing mirror, so that the trap is closed, and note that the billiard inside preserves $\omega=d(p\cdot du)$, and also $\omega^{n-1}$, so by Poincaré's recurrence theorem, there is a trajectory of some state in $E$ that returns arbitrarily close to $E$, i.e. collides with $\partial B_R$, which implies that $U$ was never trapped, since this constitutes a ray that stretches outside of the trap in the system where $\partial B_R$ is not a mirror, a contradiction. A very similar argument is given in \cite[p.~26]{TabBook}.
\end{proof}

\begin{corollary}\label{cor:upper}
Every trapped family $F\subset L_n$ has
$
\dim F\le2n-3.
$
\end{corollary}

\begin{proof}
Any embedded submanifold of dimension $2n-2$ is open in $L_n$, and thus the existence of such a trapped submanifold is forbidden by Proposition~\ref{prop:no-open}.
\end{proof}

\section{Trapping a $2n-3$ parameter family of rays}\label{sec:waist}
Here we describe the trapping of rays for $n\geq3$. Our trap, the simplest construction of which is given in Figure 1, resembles an hourglass with open ends that narrows to a central waist. It is formed by a surface of revolution 
\begin{equation}\label{eq:profile}
\rho\in C^\infty(\R,(0,\infty)),
\quad
r:=\rho(0)>0,
\quad
\rho'(0)=0,
\quad
\beta:=\rho''(0)>0.
\end{equation}
Thus
\begin{equation}\label{eq:profile-expansion}
\rho(x)=r+\frac{\beta}{2}x^2+O(x^3),
\quad
\rho'(x)=\beta x+O(x^2).
\end{equation}

Set
\[
\cM_\infty=\{(x,y):|y|=\rho(x)\},
\]
and, for $a>0$,
\[
\cM_a=\{(x,y):|y|=\rho(x),\ |x|\le a\}.
\]        
Let $f$ denote the billiard map for $\cM_\infty$.

\begin{figure}[ht]
  \centering
  \includegraphics[width=0.8\linewidth]{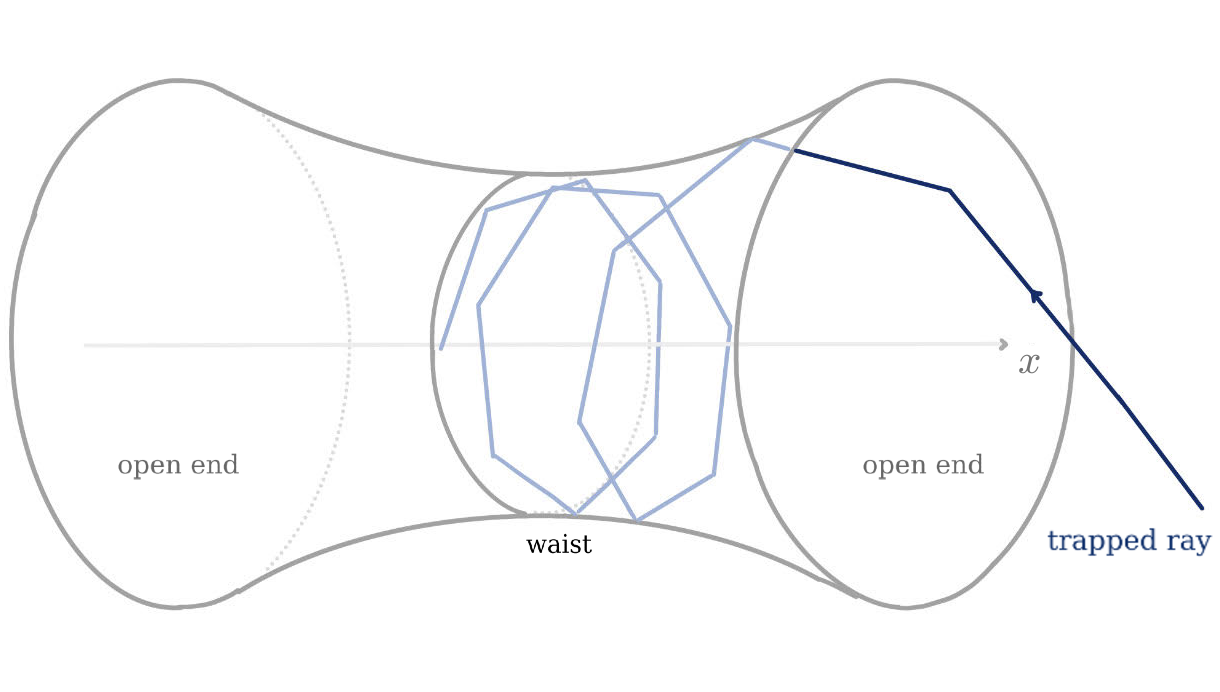}
  \caption{A ray trapped in 
  $\rho(x)=1+\frac{\beta}{2}x^2$.}
  \label{fig:hourglass}
\end{figure}

\subsection{The central billiard}
Since $\rho'(0)=0$, the normal to $\cM_\infty$ at $x=0$ has no $x$-component, so $H=\{x=0\}$ is invariant under billiard flow, and the billiard flow in $H$ is the billiard in a ball of radius $r$, where collisions are given by
\[
e\in S^{n-2},
\quad
v=(0,u),\quad u\in S^{n-2},
\quad
\inner{e}{u}<0
\]
where $re$ is the impact point on the mirror of a ray. Define $c=-\inner{e}{u}\in(0,1)$ and note both that the distance between collisions is $2rc$ and that $c$ is conserved.
Define the set
\[
\cN=
\{x=0,\ \xi=0,\ e,u\in S^{n-2},\ 0<-\inner{e}{u}<1\},
\]
where $\xi$ is velocity in $x$, and note that the invariance of $H$ under billiard flow and the conservation of $c$ imply the invariance of $\cN$ under $f$. For a compact interval $I\Subset(0,1)$ let $\cN_I=\{z\in\cN:c(z)\in I\}.$ Then $\dim\cN=\dim\cN_I=2n-4.$ Write
\[
u=-ce+s b,
\quad
s=\sqrt{1-c^2},
\quad
b\perp e,
\quad
|b|=1.
\]
At the next collision,
\[
e_+=(1-2c^2)e+2cs\,b,\quad
b_+=-2cs\,e+(1-2c^2)b,\quad
c_+=c.
\]
Thus the map $f|_\cN$ rotates $(e,b)$ through some angle $\psi(c)$ depending only on $c$. For fixed c, $D_{(e,b)}f^{k}$ is uniformly bounded in $k$, while differentiating with respect to $c$ produces at most linear growth in $|k|$. Consequently, for some constant $C_I$,
\begin{equation}\label{eq:tangent-growth}
\|D(f^k|_{\cN_I})\|+\|D(f^{-k}|_{\cN_I})\|
\le C_I(1+|k|),
\quad k\in\mathbb Z.
\end{equation}

\subsection{The linearization of $f$}
Near $\cN_I$, write an impact point and outgoing velocity as
\[
q=(x,\rho(x)e),
\quad
v=(\xi,w).
\]

\begin{lemma}\label{lem:linearization}
At a point of $\cN_I$ with parameter $c$,
\[
Df|_{(x,\xi)}=A(c)=
\begin{pmatrix}
1&2rc\\
2\beta c&1+4r\beta c^2
\end{pmatrix},
\]
with eigenvalues 
\[
\lambda_\pm(c)
=1+2r\beta c^2
 \pm2c\sqrt{r\beta+r^2\beta^2c^2},
\]
satisfying
\[
0<\lambda_-(c)<1<\lambda_+(c),
\quad
\lambda_-(c)\lambda_+(c)=1,
\]
with stable eigendirection 
\[
\xi=-\alpha(c)x,
\quad
\alpha(c)=\sqrt{\frac{\beta}{r}+\beta^2c^2}-\beta c>0.
\]
\end{lemma}

\begin{proof}
For variations in $x$ and $\xi$, take $w=\sqrt{1-\xi^2}\,u$. The time to the next collision $t$ satisfies
\[
|\rho(x)e+t w|^2=\rho(x+t\xi)^2.
\]
At $(x,\xi,t)=(0,0,2rc)$,
\[
D_{(x,\xi)}t=0,
\quad
x_1=x+2rc\,\xi+O\bigl((|x|+|\xi|)^2\bigr).
\]
Moreover,
\[
e_1=\frac{\rho(x)e+t w}{\rho(x_1)},
\quad
D_{(x,\xi)}e_1=0.
\]
By \eqref{eq:profile-expansion},
\[
\nu(x,e)
=
\frac{(-\rho'(x),e)}{\sqrt{1+\rho'(x)^2}}
=(-\beta x,e)+O(x^2).
\]
By reflection $D_{(x,\xi)}w_1=0$, so the plane $(x,\xi)$ is invariant under $Df$. At the next impact, $\inner{v^-}{\nu_1}=c+O(|x|+|\xi|),$ and $(\nu_1)_x=-\beta x_1+O(x_1^2).$ Hence
\begin{align*}
\xi_1
&=\xi-2\inner{v^-}{\nu_1}(\nu_1)_x=\xi+2\beta c\,x_1+O\bigl((|x|+|\xi|)^2\bigr)
=2\beta c\,x+(1+4r\beta c^2)\xi
  +O\bigl((|x|+|\xi|)^2\bigr).
\end{align*}
Thus $Df|_{(x,\xi)}=A(c)$. The stable eigendirections and eigenvalues of $A(c)$ follow simply since  
\[
\det A(c)=1,
\quad
\operatorname{tr}A(c)=2+4r\beta c^2>2.
\]
\end{proof}

\section{Stable manifolds and trapping rays}\label{sec:stable}

\begin{proposition}\label{prop:stable-manifold}
For every compact interval $I\Subset(0,1)$, $\cN_I$ is a normally hyperbolic invariant manifold, and is contained in a stable manifold $W^s_{\mathrm{loc}}(\cN_I)$ of dimension $2n-3$, which has a stable direction in $(x,\xi)$ along which rays converge to $\cN_I$. For $z\in W^s_{\mathrm{loc}}(\cN_I)$, we write $c(z)$ for the $c$ of the orbit in $\cN_I$ that $z$ converges to. On the submanifold of $W^s_{\mathrm{loc}}(\cN_I)$ where $x>0$,
\begin{equation}\label{eq:stable-graph}
\xi=-\alpha(c)x+O(x^2),
\end{equation}
uniformly over $\cN_I$.
\end{proposition}

\begin{proof}
Choose a compact interval $J$ such that $I\Subset J\Subset(0,1)$. Then on $\cN_J$, $c$ is bounded below by a positive constant, so $f$ and $f^{-1}$ are $C^\infty$ on a neighborhood of $\cN_J$ in the collision space of $\cM_\infty$.

Note that Lemma~\ref{lem:linearization} and the conservation of $c$ imply that $Df|_{(x,\xi)}=A(c)$ at every impact of a given orbit, so $Df^k|_{(x,\xi)}=A(c)^k$. The stable and unstable eigendirections of $A(c)$  define subbundles \(E^s\) and \(E^u\) over \(\cN_J\). Note also that the compactness of \(J\) implies that the contraction on \(E^s\) and the expansion on \(E^u\) are uniform (at rates $\lambda_s<1<\lambda_u$), and by \eqref{eq:tangent-growth}, $\|D(f^{\pm k}|_{\cN_J})\|$ grows at most linearly in $k$. Thus the exponential rates in the normal directions dominate the at most linear growth in the tangential directions, i.e. \(\cN_J\) is \(\ell\)-normally hyperbolic for every finite \(\ell\). We may modify $f$ near $\partial\cN_J$ (away from $\cN_I$), to make $\cN_J$ inflowing. A stable manifold theorem for normally hyperbolic invariant manifolds~\cite[Theorem 4]{Fenichel} (applied to $f^{-1}$) then gives a
\(C^\infty\) local stable manifold which, restricted to \(\cN_I\), gives the stable manifold
$W^s_{\mathrm{loc}}(\cN_I)$ for $f$. For $z\in W^s_{\mathrm{loc}}(\cN_I)$ with an offset $x_k$ at the $k$th impact, $c$ changes by $O(x_k)$ and the angle changes by $\psi(c_k)+O(x_k)$, and $\sum_k x_k<\infty$ since $x_k\leq C\lambda^k_s$, so the orbit of $z$ converges to a single orbit in $\cN_I$ and $c(z)$ is well defined. Since $E^s$ is one dimensional, $\dim W^s_{\mathrm{loc}}(\cN_I) = \dim \cN_I+1 = 2n-3.$ Finally, the stable eigenvector of $A(c)$ in Lemma~\ref{lem:linearization} has a nonzero $x$ component, which gives \eqref{eq:stable-graph} on \(W^{s,+}_{\mathrm{loc}}(\cN_I)\).
\end{proof}

\begin{proposition}\label{prop:trapping}
For $n\geq3$ and all sufficiently small $a>0$, the mirror $\cM_a$ traps a family of rays of dimension $2n-3$.
\end{proposition}

\begin{proof}
Fix \(I\Subset(0,1)\), and let
\(W^{s,+}_{\mathrm{loc}}(\cN_I)\) denote \(W^{s}_{\mathrm{loc}}(\cN_I)\) where $x>0$. Lemma~4.1 and Proposition~5.1 imply that (after shrinking \(W^{s,+}_{\mathrm{loc}}(\cN_I)\) if necessary) $0<x(fz)<x(z)$ for every \(z\in W^{s,+}_{\mathrm{loc}}(\cN_I)\).
For sufficiently small \(a>0\) define 
\[
    S_a
    =
    \left\{
        z\in W^{s,+}_{\mathrm{loc}}(\cN_I):
        c(z)\in I^\circ,\quad
        x(z)<a<x(f^{-1}z)
    \right\}.
\]
of dimension $\dim S_a=\dim W^s_{\mathrm{loc}}(\cN_I)=2n-3$.

Then for \(z\in S_a\), the ray from \(f^{-1}z\) to \(z\) crosses
the plane \(x=a\) through the aperture of \(\cM_a\). Note that when we truncate to $|x|=a$, there is no longer a collision at \(f^{-1}z\), and the trajectory arrives at $z$ from $x=\infty$. Note that for $k\geq0$, $0<x(f^kz)<a$ and by Proposition~5.1, \(f^kz\) converges to \(\cN_I\). Indeed, since $x$ is affine between collisions, the forward trajectory exists solely within \(0<x<a\). One can easily see that the trajectory is nonsingular. The map $\Psi$ from collision states to their incoming oriented lines is a local diffeomorphism, and $\Psi(z_1)=\Psi(z_2)$ implies that $z_1$ and $z_2$ are the first collisions of the same incoming ray, which are necessarily unique, so $\Psi$ is injective on $S_a$ and $\Psi|_{S_{a}}$ is a diffeomorphism onto $\Psi(S_a)$, which is a trapped family of dimension $2n-3$. Hence Theorem~\ref{thm:main} is proved (for $n\geq 3$, $n=2$ is treated in Remark~\ref{rem:planar}).
\end{proof}

\section{A trapped non-normal 2-parameter family of rays in $\R^3$}\label{sec:nonnormal}
For $(u,p)\in L_n$, set $\omega=d(p\cdot du)$. For a family $q(s,t)$, $u(s,t)$ where $q$ is a point on each line and $u$ is its unit direction, $\omega(\partial_s,\partial_t)=q_s\cdot u_t-q_t\cdot u_s.$
\begin{lemma}\label{lem:normal-isotropic}
If a 2-parameter family $U\subset L_3$ is locally normal to a smooth surface, then $\omega|_U=0$.
\end{lemma}
\begin{proof}
Choose local coordinates $(s,t)$ on the surface, let $q(s,t)$ be its position vector, and let $u(s,t)$ be its unit normal.
Differentiating $u\cdot q_s=u\cdot q_t=0$ and using $q_{st}=q_{ts}$ gives $q_s\cdot u_t=q_t\cdot u_s$, so $\omega(\partial_s,\partial_t)=0$.
\end{proof}
Now take $n=3$, where the billiard at $x=0$ is the billiard in the disk of radius $r$.
With $e(\theta)=(\cos\theta,\sin\theta)$ and $Je(\theta)=(-\sin\theta,\cos\theta)$, at the impact point $q_0(\theta,c)=(0,r e(\theta)),$ note that the direction $u_0(\theta,c)=(0,ce(\theta)+s\,Je(\theta))$
has reflection $(0,-ce(\theta)+s\,Je(\theta))$ in $\cN$,  so $(q_0,u_0)$ gives all incoming lines at $x=0$.
Then we can easily see that $\omega$ is
\begin{equation}\label{eq:central-omega}
q_{0,\theta}\cdot u_{0,c}-q_{0,c}\cdot u_{0,\theta}
=-\frac{rc}{\sqrt{1-c^2}}\ne0.
\end{equation}
\begin{proposition}\label{prop:nonnormal}
For all sufficiently small $a>0$, $\cM_a\subset\R^3$ traps a non-normal 2-parameter family of rays.
\end{proposition}
\begin{proof}
Fix $I\Subset(0,1)$ and a chart $P\Subset\cN_I^\circ$ with coordinates $(\theta,c)$. Let
\[
\Lambda=\min_{c\in I}\lambda_+(c)>1,
\quad
\Lambda^{-1}<\vartheta<1,
\quad
\eta=\vartheta a.
\]
For small $\eta\ge0$, let $j_\eta:P\longrightarrow W^s_{\mathrm{loc}}(\cN_I)$ send each $z$ to the point of the stable manifold over $z$ with $x=\eta$. Then $j_0$ is the inclusion of $P$ and $\xi=-\alpha(c)\eta+O(\eta^2).$ Set $G_\eta=\Psi\circ j_\eta:P\longrightarrow L_3$
and note that because $G_\eta$ is a composition of smooth functions, \eqref{eq:central-omega} implies that on P, $G_\eta^*\omega$ is nowhere zero for all sufficiently small $\eta\ge0$.
We then shrink $P$ so that $G_\eta$ is an embedding whose image is non-normal by Lemma~\ref{lem:normal-isotropic}. For all sufficiently small $a>0$, the proof of Proposition~\ref{prop:trapping} applies to $j_{\vartheta a}(P)$, since
\[
0<\vartheta a<a,
\quad
x(f^{-1}j_{\vartheta a}(z))
=\lambda_+(c)\vartheta a+O(a^2)>a.
\]
Thus every ray in $G_{\vartheta a}$ enters the trap through the aperture and is trapped, hence Theorem~\ref{thm:nonnormal} is proved.
\end{proof}

\begin{remark}
More simply, the existence of a trapped non-normal 2-parameter family follows when we note that the trapped family $\Psi(S_a)$ of Proposition~\ref{prop:trapping} is a hypersurface in $L_3$, so in suitable local coordinates $(q_1,p_1,q_2,p_2)$ we have $\omega=dp_1\wedge dq_1+dp_2\wedge dq_2$ and $\Psi(S_a)$ is given by $\{p_2=0\}$. Then $\{p_2=q_2=0\}$ is a trapped 2-parameter family on which $\omega|_{\{p_2=q_2=0\}}=dp_1\wedge dq_1\neq0$, so by
Lemma~\ref{lem:normal-isotropic} it is not locally normal to any smooth surface.
\end{remark}

\begin{remark}\label{rem:planar}
We can trap a 1-parameter family in $n=2$ in $y=\pm\rho(x),\quad |x|\leq a$ by applying the computation of Lemma~\ref{lem:linearization} at $c=1$ (the vertical line at $x=0$), and noting that $\det A(1)=1$ and $\mathrm{tr} (A(1))>2$ making the orbit hyperbolic, with a one dimensional stable manifold amenable to the argument given in Proposition~\ref{prop:trapping}, which yields a trapped family of dimension $1=2n-3.$ To first order and for a trap with a parabolic profile, we can trap a ray if we aim it at the points $(0,\pm\sqrt{r^2+\frac{r}{\beta}})$. Indeed, note (where $\alpha=\alpha(1)$) that reflecting back the trapped direction $\xi^+=-\alpha x+O(x^3)$ at $(x,\rho(x))$ yields $\xi^-=-(\alpha+2\beta)x+O(x^3)$, so the line representing the initial trapped ray crosses $x=0$ at height $r+1/(\alpha+2\beta)+O(x^2)=\sqrt{r^2+\frac{r}{\beta}}+O(x^2)$.
\end{remark}

\section*{Acknowledgments}
The author would like to thank Serge Tabachnikov for his correspondence.

\end{document}